\documentclass{birkjour}
 \newtheorem{thm}{Theorem}[section]
 \newtheorem{cor}[thm]{Corollary}
 \newtheorem{lem}[thm]{Lemma}
 
 \theoremstyle{definition}
 \newtheorem{defn}[thm]{Definition}
 \theoremstyle{remark}

 \numberwithin{equation}{section}

\newcommand{\mymod}[3]{#1 \equiv #2 \kern -0.5em \pmod{#3}}
\newcommand{\mynotmod}[3]{#1 \not \equiv #2 \kern -0.6em \pmod{#3}}
\renewcommand{\theequation}{\thesection.\arabic{equation}}

\begin{document}

%
%
%
%
%
%
%
%
%

\title[On the third-order Horadam matrix sequences]{On the third-order Horadam matrix sequences}

\author[G. Cerda-Morales]{Gamaliel Cerda-Morales}
\address{Instituto de Matem\'aticas, Pontificia Universidad Cat\'olica de Valpara\'iso, Blanco Viel 596, Valpara\'iso, Chile.}
\email{gamaliel.cerda.m@mail.pucv.cl}




\begin{abstract}
In this paper, we first give new generalizations for third-order Horadam $\{H_{n}^{(3)}\}_{n\in \mathbb{N}}$ and generalized Tribonacci $\{h_{n}^{(3)}\}_{n\in \mathbb{N}}$ sequences for classic Horadam and generalized Fibonacci numbers. Considering these sequences, we define the matrix sequences which have elements of $\{H_{n}^{(3)}\}_{n\in \mathbb{N}}$ and $\{h_{n}^{(3)}\}_{n\in \mathbb{N}}$. Then we investigate their properties.
\end{abstract}

\maketitle
\section {Introduction}


The Horadam numbers have many interesting properties and applications in many fields of science (see, e.g., \cite{Lar1,Lar2}). The Horadam numbers $H_{n}(a,b;r,s)$ or $H_{n}$ are defined by the recurrence relation
\begin{equation}\label{e1}
H_{0}=a,\ H_{1}=b,\ H_{n+2}=rH_{n+1}+sH_{n},\ n\geq0.
\end{equation}
Another important sequence is the generalized Fibonacci sequence $\{h_{n}^{(3)}\}_{n\in \mathbb{N}}$.  This sequence is defined by the recurrence relation $h_{n+2}=rh_{n+1}+sh_{n}$, with $h_{0}=0$, $h_{1}=1$ and $n\geq0$. (see, \cite{Ci1}).

In \cite{Sha,Wa} the Horadam recurrence relation (\ref{e1}) is extended to higher order recurrence relations and the basic list of identities provided by A. F. Horadam is expanded and extended to several identities for some of the higher order cases. In fact, third-order Horadam numbers, $\{H_{n}^{(3)}\}_{n\geq0}$, and generalized Tribonacci numbers, $\{h_{n}^{(3)}\}_{n\geq0}$, are defined by
\begin{equation}\label{ec:5}
H_{n+3}^{(3)}=rH_{n+2}^{(3)}+sH_{n+1}^{(3)}+tH_{n}^{(3)},\ H_{0}^{(3)}=a,\ H_{1}^{(3)}=b,\ H_{2}^{(3)}=c,\ n\geq0,
\end{equation}
and 
\begin{equation}\label{ec:6}
h_{n+3}^{(3)}=rh_{n+2}^{(3)}+sh_{n+1}^{(3)}+th_{n}^{(3)},\ h_{0}^{(3)}=0,\ h_{1}^{(3)}=1,\ h_{2}^{(3)}=r,\ n\geq0,
\end{equation}
respectively.

Some of the following properties given for third-order Horadam numbers and generalized Tribonacci numbers are revisited in this paper (for more details, see \cite{Cer1,Sha,Wa}). 
\begin{equation}\label{e4}
H_{n+m}^{(3)}=h_{n}^{(3)}H_{m+1}^{(3)}+\left(sh_{n-1}^{(3)}+th_{n-2}^{(3)}\right)H_{m}^{(3)}+th_{n-1}^{(3)}H_{m-1}^{(3)},
\end{equation}
\begin{equation}\label{e5}
\left(h_{n}^{(3)}\right)^{2}+s\left(h_{n-1}^{(3)}\right)^{2}+2th_{n-1}^{(3)}h_{n-2}^{(3)}=h_{2n-1}^{(3)}
\end{equation}
and
\begin{equation}\label{ec5}
\left(H_{n}^{(3)}\right)^{2}+s\left(H_{n-1}^{(3)}\right)^{2}+2tH_{n-1}^{(3)}H_{n-2}^{(3)}=\left\lbrace 
\begin{array}{c}
cH_{2n-2}^{(3)}+\left(sb+ta\right)H_{2n-3}^{(3)}\\
+tbH_{2n-4}^{(3)},
\end{array}
\right\rbrace,
\end{equation}
where $n\geq 2$ and $m\geq 1$.

As the elements of this Tribonacci-type number sequence provide third order iterative relation, its characteristic equation is $x^{3}-rx^{2}-sx-t=0$, whose roots are $\alpha=\frac{r}{3}+A+B$, $\omega_{1}=\frac{r}{3}+\epsilon A+\epsilon^{2} B$ and $\omega_{2}=\frac{r}{3}+\epsilon^{2}A+\epsilon B$, where $$A=\sqrt[3]{\frac{r^{3}}{27}+\frac{rs}{6}+\frac{t}{2}+\sqrt{\Delta}},\ B=\sqrt[3]{\frac{r^{3}}{27}+\frac{rs}{6}+\frac{t}{2}-\sqrt{\Delta}},$$ with $\Delta=\Delta(r,s,t)=\frac{r^{3}t}{27}-\frac{r^{2}s^{2}}{108}+\frac{rst}{6}-\frac{s^{3}}{27}+\frac{t^{2}}{4}$ and $\epsilon=-\frac{1}{2}+\frac{i\sqrt{3}}{2}$. 

In this paper, $\Delta>0$, then the cubic equation $x^{3}-rx^{2}-sx-t=0$ has one real and two nonreal solutions, the latter being conjugate complex. Thus, the Binet formula for the third-order Horadam numbers can be expressed as:
\begin{equation}\label{eq:8}
H_{n}^{(3)}=\frac{P\alpha^{n}}{(\alpha-\omega_{1})(\alpha-\omega_{2})}-\frac{Q\omega_{1}^{n}}{(\alpha-\omega_{1})(\omega_{1}-\omega_{2})}+\frac{R\omega_{2}^{n}}{(\alpha-\omega_{2})(\omega_{1}-\omega_{2})},
\end{equation}
where the coefficients are $P=c-(\omega_{1}+\omega_{2})b+\omega_{1}\omega_{2}a$, $Q=c-(\alpha+\omega_{2})b+\alpha\omega_{2}a$ and $R=c-(\alpha+\omega_{1})b+\alpha\omega_{1}a$.

In particular, if $a=0$, $b=1$ and $c=r$, we obtain $H_{n}^{(3)}=h_{n}^{(3)}$. In this case, $P=\alpha$, $Q=\omega_{1}$ and $R=\omega_{2}$ in Eq. (\ref{eq:8}). In fact, the generalized Tribonacci sequence is the generalization of the well-known sequences like Tribonacci, Padovan, Narayana and third-order Jacobsthal. For example, $\{H_{n}(0,1,1;1,1,1)\}_{n\geq0}$, $\{H_{n}(0,1,0;0,1,1)\}_{n\geq0}$, are Tribonacci and Padovan sequences, respectively.

In \cite{Ci1,Ci2}, the authors defined a new matrix generalization of the Fibonacci and Lucas numbers, and using essentially a matrix approach they showed properties of these matrix sequences. The main motivation of this article is to study the matrix sequences of third-order Horadam sequence and generalized Tribonacci sequence.


\section {The third-order Horadam, generalized Tribonacci sequences and their matrix sequences}

Now, considering these sequences, we define the matrix sequences which have elements of third-order Horadam and generalized Tribonacci sequences.

\begin{defn}
The third-order Horadam matrix sequence $\{M_{H,n}^{(3)}\}_{n\in \mathbb{N}}$ and generalized Tribonacci matrix sequence $\{M_{h,n}^{(3)}\}_{n\in \mathbb{N}}$ are defined respectively by
\begin{equation}\label{d1}
M_{H,n+3}^{(3)}=rM_{H,n+2}^{(3)}+sM_{H,n+1}^{(3)}+tM_{H,n}^{(3)},\ n\geq 0,
\end{equation}
\begin{equation}\label{d2}
M_{h,n+3}^{(3)}=rM_{h,n+2}^{(3)}+sM_{h,n+1}^{(3)}+tM_{h,n}^{(3)},\ n\geq 0,
\end{equation}
with initial conditions $$M_{H,0}^{(3)}=\left[
\begin{array}{ccc}
b& c-rb& ta \\ 
a&b-ra& c-rb-sa\\ 
\frac{1}{t}(c-rb-sa) & \frac{r}{t}(-c+rb+sa)&\frac{1}{t}\left\lbrace \begin{array}{c}-sc+(t+rs)b\\+(s^{2}-rt)a\end{array}\right\rbrace
\end{array}%
\right],$$  $$M_{H,1}^{(3)}=\left[
\begin{array}{ccc}
c& sb+ta& tb \\ 
b&c-rb& ta\\ 
a & b-ra&c-rb-sa
\end{array}%
\right],\ M_{H,2}^{(3)}=\left[
\begin{array}{ccc}
rc+sb+ta& sc+tb& tc \\ 
c &sb+ta& tb\\ 
b & c-rb&ta%
\end{array}%
\right]$$
and $$M_{h,0}^{(3)}=\left[
\begin{array}{ccc}
1& 0& 0 \\ 
0&1& 0\\ 
0 & 0&1%
\end{array}%
\right],\  M_{h,1}^{(3)}=\left[
\begin{array}{ccc}
r& s& t \\ 
1&0& 0\\ 
0 & 1&0%
\end{array}%
\right],\ M_{h,2}^{(3)}=\left[
\begin{array}{ccc}
r^{2}+s& rs+t& rt \\ 
r&s& t\\ 
1 & 0&0%
\end{array}%
\right].$$
\end{defn}

\begin{thm}\label{t1}
For $n\geq 0$, we have
\begin{equation}\label{p1}
\begin{aligned}
M_{H,n}^{(3)}&=\left(\frac{M_{H,2}^{(3)}-(\omega_{1}+\omega_{2})M_{H,1}^{(3)}+\omega_{1}\omega_{2}M_{H,0}^{(3)}}{(\alpha-\omega_{1})(\alpha-\omega_{2})}\right)\alpha^{n}\\
&\ \ -\left(\frac{M_{H,2}^{(3)}-(\alpha+\omega_{2})M_{H,1}^{(3)}+\alpha\omega_{2}M_{H,0}^{(3)}}{(\alpha-\omega_{1})(\omega_{1}-\omega_{2})}\right)\omega_{1}^{n}\\
&\ \ + \left(\frac{M_{H,2}^{(3)}-(\alpha+\omega_{1})M_{H,1}^{(3)}+\alpha\omega_{1}M_{H,0}^{(3)}}{(\alpha-\omega_{2})(\omega_{1}-\omega_{2})}\right)\omega_{2}^{n}
\end{aligned}
\end{equation}
and
\begin{equation}\label{p2}
\begin{aligned}
M_{h,n}^{(3)}&=\left(\frac{M_{h,2}^{(3)}-(\omega_{1}+\omega_{2})M_{h,1}^{(3)}+\omega_{1}\omega_{2}M_{h,0}^{(3)}}{(\alpha-\omega_{1})(\alpha-\omega_{2})}\right)\alpha^{n}\\
&\ \ -\left(\frac{M_{h,2}^{(3)}-(\alpha+\omega_{2})M_{h,1}^{(3)}+\alpha\omega_{2}M_{h,0}^{(3)}}{(\alpha-\omega_{1})(\omega_{1}-\omega_{2})}\right)\omega_{1}^{n}\\
&\ \ + \left(\frac{M_{h,2}^{(3)}-(\alpha+\omega_{1})M_{h,1}^{(3)}+\alpha\omega_{1}M_{h,0}^{(3)}}{(\alpha-\omega_{2})(\omega_{1}-\omega_{2})}\right)\omega_{2}^{n}.
\end{aligned}
\end{equation}
\end{thm}
\begin{proof}
(\ref{p1}): The solution of Eq. (\ref{d1}) is
\begin{equation}\label{e13}
M_{H,n}^{(3)}=A_{H}\alpha^{n}+B_{H}\omega_{1}^{n}+C_{H}\omega_{2}^{n},
\end{equation}
with $A_{H}, B_{H}, C_{H}\in \mathbb{C}$. Then, let $M_{H,0}^{(3)}=A_{H}+B_{H}+C_{H}$, $M_{H,1}^{(3)}=A_{H}\alpha+B_{H}\omega_{1}+C_{H}\omega_{2}$ and $M_{H,2}^{(3)}=A_{H}\alpha^{2}+B_{H}\omega_{1}^{2}+C_{H}\omega_{2}^{2}$. Therefore, we have $(\alpha-\omega_{1})(\alpha-\omega_{2})A_{H}=M_{H,2}^{(3)}-(\omega_{1}+\omega_{2})M_{H,1}^{(3)}+\omega_{1}\omega_{2}M_{H,0}^{(3)}$, $(\omega_{1}-\alpha)(\omega_{1}-\omega_{2})B_{H}=M_{H,2}^{(3)}-(\alpha+\omega_{2})M_{H,1}^{(3)}+\alpha\omega_{2}M_{H,0}^{(3)}$ and $(\alpha-\omega_{2})(\omega_{1}-\omega_{2})C_{H}=M_{H,2}^{(3)}-(\alpha+\omega_{1})M_{H,1}^{(3)}+\alpha\omega_{1}M_{H,0}^{(3)}$. Using $A_{H}$, $B_{H}$ and $C_{H}$ in Eq. (\ref{e13}), we obtain
\begin{align*}
M_{H,n}^{(3)}&=\left(\frac{M_{H,2}^{(3)}-(\omega_{1}+\omega_{2})M_{H,1}^{(3)}+\omega_{1}\omega_{2}M_{H,0}^{(3)}}{(\alpha-\omega_{1})(\alpha-\omega_{2})}\right)\alpha^{n}\\
&\ \ -\left(\frac{M_{H,2}^{(3)}-(\alpha+\omega_{2})M_{H,1}^{(3)}+\alpha\omega_{2}M_{H,0}^{(3)}}{(\alpha-\omega_{1})(\omega_{1}-\omega_{2})}\right)\omega_{1}^{n}\\
&\ \ + \left(\frac{M_{H,2}^{(3)}-(\alpha+\omega_{1})M_{H,1}^{(3)}+\alpha\omega_{1}M_{H,0}^{(3)}}{(\alpha-\omega_{2})(\omega_{1}-\omega_{2})}\right)\omega_{2}^{n}.
\end{align*}
(\ref{p2}): The proof is similar to the proof of (\ref{p1}).
\end{proof}

The following theorem gives us the $n$-th general term of the sequence given in (\ref{d1}) and (\ref{d2}).
\begin{thm}\label{t2}
For $n\geq3$, we have
\begin{equation}\label{p3}
M_{H,n}^{(3)}=\left[
\begin{array}{ccc}
H_{n+1}^{(3)}& sH_{n}^{(3)}+tH_{n-1}^{(3)}& tH_{n}^{(3)} \\ 
H_{n}^{(3)}&sH_{n-1}^{(3)}+tH_{n-2}^{(3)}& tH_{n-1}^{(3)}\\ 
H_{n-1}^{(3)} & sH_{n-2}^{(3)}+tH_{n-3}^{(3)}&tH_{n-2}^{(3)}
\end{array}
\right]
\end{equation}
and
\begin{equation}\label{p4}
M_{h,n}^{(3)}=\left[
\begin{array}{ccc}
h_{n+1}^{(3)}& sh_{n}^{(3)}+th_{n-1}^{(3)}& th_{n}^{(3)} \\ 
h_{n}^{(3)}&sh_{n-1}^{(3)}+th_{n-2}^{(3)}& th_{n-1}^{(3)}\\ 
h_{n-1}^{(3)} & sh_{n-2}^{(3)}+th_{n-3}^{(3)}& th_{n-2}^{(3)}
\end{array}
\right].
\end{equation}
\end{thm}
\begin{proof}
(\ref{p3}): Let use the principle of mathematical induction on $n$.
Let us consider $n=0$ in (\ref{ec:5}). We have $tH_{-1}^{(3)}=c-rb-sa$, $t^{2}H_{-2}^{(3)}=-sc+(t+rs)b+(s^{2}-rt)a$ and $t^{3}H_{-3}^{(3)}=(s^{2}-rt)c+(r^{2}t-st-rs^{2})b+(t^{2}-s^{3}+2rst)a$. Then, we write $M_{H,0}^{(3)}$ as in Eq.(\ref{d1}). By iterating this procedure and considering induction steps, let us assume that the equality in (\ref{p3}) holds for all $n\leq k$. 

To finish the proof, we have to show that (\ref{p3}) also holds for $n=k+1$ by considering (\ref{ec:5}) and (\ref{d1}). Therefore we get
\begin{align*}
&M_{H,k+1}^{(3)}\\
&=rM_{H,k}^{(3)}+sM_{H,k-1}^{(3)}+tM_{H,k-2}^{(3)}\\
&=\left[
\begin{array}{ccc}
rH_{k+1}^{(3)}+sH_{k}^{(3)}+tH_{k-1}^{(3)}& sH_{k+1}^{(3)}+tH_{k}^{(3)}& rtH_{k}^{(3)}+stH_{k-1}^{(3)}+t^{2}H_{k-2}^{(3)}\\ 
rH_{k}^{(3)}+sH_{k-1}^{(3)}+tH_{k-2}^{(3)}&sH_{k}^{(3)}+tH_{k-1}^{(3)}& rtH_{k-1}^{(3)}+stH_{k-2}^{(3)}+t^{2}H_{k-3}^{(3)}\\ 
rH_{k-1}^{(3)}+sH_{k-2}^{(3)}+tH_{k-3}^{(3)}&sH_{k-1}^{(3)}+tH_{k-2}^{(3)}&rtH_{k-2}^{(3)}+stH_{k-3}^{(3)}+t^{2}H_{k-4}^{(3)}
\end{array}%
\right]\\
&=\left[
\begin{array}{ccc}
H_{k+2}^{(3)}& sH_{k+1}^{(3)}+tH_{k}^{(3)}& tH_{k+1}^{(3)} \\ 
H_{k+1}^{(3)}&sH_{k}^{(3)}+tH_{k-1}^{(3)}& tH_{k}^{(3)}\\ 
H_{k}^{(3)} & sH_{k-1}^{(3)}+tH_{k-2}^{(3)}& tH_{k-1}^{(3)}%
\end{array}%
\right].
\end{align*}
Hence we obtain the result.
If a similar argument is applied to (\ref{p4}), the proof is clearly seen.
\end{proof}

\begin{thm}\label{t3}
Assume that $x\neq 0$. We obtain,
\begin{equation}\label{p5}
\sum_{k=0}^{n}\frac{M_{H,k}^{(3)}}{x^{k}}=\frac{1}{x^{n}\nu(x)}\left\lbrace
\begin{array}{c}
x^{n+1}\left\lbrace \begin{array}{c} M_{H,2}^{(3)}-rM_{H,1}^{(3)}-sM_{H,0}^{(3)}\\ +\left(M_{H,1}^{(3)}-rM_{H,0}^{(3)}\right)x\\+M_{H,0}^{(3)}x^{2}\end{array}\right\rbrace \\
-tM_{H,n}^{(3)}-\left(M_{H,n+2}^{(3)}-rM_{H,n+1}^{(3)}\right)x\\-M_{H,n+1}^{(3)}x^{2}
\end{array}
\right\rbrace
\end{equation}
and
\begin{equation}\label{p6}
\sum_{k=0}^{n}\frac{M_{h,k}^{(3)}}{x^{k}}=\frac{1}{x^{n}\nu(x)}\left\lbrace
\begin{array}{c}
x^{n+1}\left\lbrace \begin{array}{c} M_{h,2}^{(3)}-rM_{h,1}^{(3)}-sM_{h,0}^{(3)}\\ +\left(M_{h,1}^{(3)}-rM_{h,0}^{(3)}\right)x+M_{h,0}^{(3)}x^{2}\end{array}\right\rbrace \\
-tM_{h,n}^{(3)}-\left(M_{h,n+2}^{(3)}-rM_{h,n+1}^{(3)}\right)x\\
-M_{h,n+1}^{(3)}x^{2}
\end{array}
\right\rbrace,
\end{equation}
where $\nu(x)=x^{3}-rx^{2}-sx-t$.
\end{thm}
\begin{proof}
In contrast, here we will just prove (\ref{p6}) since the proof of (\ref{p5}) can be done in a similar way. From Theorem \ref{t1}, we have
\begin{align*}
\sum_{k=0}^{n}\frac{M_{h,k}^{(3)}}{x^{k}}&=\left(\frac{M_{h,2}^{(3)}-(\omega_{1}+\omega_{2})M_{h,1}^{(3)}+\omega_{1}\omega_{2}M_{h,0}^{(3)}}{(\alpha-\omega_{1})(\alpha-\omega_{2})}\right)\sum_{k=0}^{n}\left(\frac{\alpha}{x}\right)^{k}\\
&\ \ - \left(\frac{M_{h,2}^{(3)}-(\alpha+\omega_{2})M_{h,1}^{(3)}+\alpha\omega_{2}M_{h,0}^{(3)}}{(\alpha-\omega_{1})(\omega_{1}-\omega_{2})}\right)\sum_{k=0}^{n}\left(\frac{\omega_{1}}{x}\right)^{k}\\
&\  \ + \left(\frac{M_{h,2}^{(3)}-(\alpha+\omega_{1})M_{h,1}^{(3)}+\alpha\omega_{1}M_{h,0}^{(3)}}{(\alpha-\omega_{2})(\omega_{1}-\omega_{2})}\right)\sum_{k=0}^{n}\left(\frac{\omega_{2}}{x}\right)^{k}.
\end{align*}
By considering the definition of a geometric sequence, we get
\begin{align*}
\sum_{k=0}^{n}\frac{M_{h,k}^{(3)}}{x^{k}}&=\left(\frac{M_{h,2}^{(3)}-(\omega_{1}+\omega_{2})M_{h,1}^{(3)}+\omega_{1}\omega_{2}M_{h,0}^{(3)}}{(\alpha-\omega_{1})(\alpha-\omega_{2})}\right)\frac{x^{n+1}-\alpha^{n+1}}{x^{n}(x-\alpha)}\\
&\ \ - \left(\frac{M_{h,2}^{(3)}-(\alpha+\omega_{2})M_{h,1}^{(3)}+\alpha\omega_{2}M_{h,0}^{(3)}}{(\alpha-\omega_{1})(\omega_{1}-\omega_{2})}\right)\frac{x^{n+1}-\omega_{1}^{n+1}}{x^{n}(x-\omega_{1})}\\
&\  \ + \left(\frac{M_{h,2}^{(3)}-(\alpha+\omega_{1})M_{h,1}^{(3)}+\alpha\omega_{1}M_{h,0}^{(3)}}{(\alpha-\omega_{2})(\omega_{1}-\omega_{2})}\right)\frac{x^{n+1}-\omega_{2}^{n+1}}{x^{n}(x-\omega_{2})}\\
&=\frac{1}{x^{n}\nu(x)}\left\lbrace
\begin{array}{c}
A_{h}(x^{n+1}-\alpha^{n+1})(\omega_{1}-x)(\omega_{2}-x)\\
- B_{h}(x^{n+1}-\omega_{1}^{n+1})(\alpha-x)(\omega_{2}-x)\\
+ C_{h}(x^{n+1}-\omega_{2}^{n+1})(\alpha-x)(\omega_{1}-x)
\end{array}
\right\rbrace,
\end{align*}
where $$\left\lbrace
\begin{array}{c}
A_{h}=\frac{M_{h,2}^{(3)}-(\omega_{1}+\omega_{2})M_{h,1}^{(3)}+\omega_{1}\omega_{2}M_{h,0}^{(3)}}{(\alpha-\omega_{1})(\alpha-\omega_{2})},\\
B_{h}=\frac{M_{h,2}^{(3)}-(\alpha+\omega_{2})M_{h,1}^{(3)}+\alpha\omega_{2}M_{h,0}^{(3)}}{(\alpha-\omega_{1})(\omega_{1}-\omega_{2})},\ 
C_{h}=\frac{M_{h,2}^{(3)}-(\alpha+\omega_{1})M_{h,1}^{(3)}+\alpha\omega_{1}M_{h,0}^{(3)}}{(\alpha-\omega_{2})(\omega_{1}-\omega_{2})}
\end{array}
\right.
$$
and $\nu(x)=x^{3}-rx^{2}-sx-t$. 

Further, using $\alpha+\omega_{1}+\omega_{2}=r$, $\alpha\omega_{1}+\alpha\omega_{2}+\omega_{1}\omega_{2}=-s$ and $\alpha\omega_{1}\omega_{2}=t$, if we rearrange the last equality, then we obtain
\begin{align*}
\sum_{k=0}^{n}\frac{M_{h,k}^{(3)}}{x^{k}}&=\frac{1}{x^{n}\nu(x)}\left\lbrace
\begin{array}{c}
A_{h}(x^{n+1}-\alpha^{n+1})(\omega_{1}\omega_{2}-(\omega_{1}+\omega_{2})x+x^{2})\\
- B_{h}(x^{n+1}-\omega_{1}^{n+1})(\alpha\omega_{2}-(\alpha+\omega_{2})x+x^{2})\\
+ C_{h}(x^{n+1}-\omega_{2}^{n+1})(\alpha\omega_{1}-(\alpha+\omega_{1})x+x^{2})
\end{array}
\right\rbrace\\
&=\frac{1}{x^{n}\nu(x)}\left\lbrace
\begin{array}{c}
x^{n+1}\left\lbrace \begin{array}{c} A_{h}(\omega_{1}\omega_{2}-(\omega_{1}+\omega_{2})x+x^{2})\\
- B_{h}(\alpha\omega_{2}-(\alpha+\omega_{2})x+x^{2})\\
+ C_{h}(\alpha\omega_{1}-(\alpha+\omega_{1})x+x^{2})
\end{array}\right\rbrace\\
- A_{h}\alpha^{n+1}(\omega_{1}\omega_{2}-(\omega_{1}+\omega_{2})x+x^{2})\\
+ B_{h}\omega_{1}^{n+1}(\alpha\omega_{2}-(\alpha+\omega_{2})x+x^{2})\\
- C_{h}\omega_{2}^{n+1}(\alpha\omega_{1}-(\alpha+\omega_{1})x+x^{2})\\
\end{array}
\right\rbrace\\
&=\frac{1}{x^{n}\nu(x)}\left\lbrace
\begin{array}{c}
x^{n+1}\left\lbrace \begin{array}{c} M_{h,2}^{(3)}-rM_{h,1}^{(3)}-sM_{h,0}^{(3)}\\ +\left(M_{h,1}^{(3)}-rM_{h,0}^{(3)}\right)x+M_{h,0}^{(3)}x^{2}\end{array}\right\rbrace \\
-tM_{h,n}^{(3)}-\left(M_{h,n+2}^{(3)}-rM_{h,n+1}^{(3)}\right)x-M_{h,n+1}^{(3)}x^{2}
\end{array}
\right\rbrace.
\end{align*}
So, the proof is completed.
\end{proof}

In the following theorem, we give the sum of third-order Horadam and generalized Tribonacci matrix sequences corresponding to different indices.
\begin{thm}\label{t4}
For $l\geq m$, we have
\begin{equation}\label{p7}
\sum_{k=0}^{n}M_{H,mk+l}^{(3)}=\frac{1}{\sigma_{m}}\left\lbrace
\begin{array}{c}
M_{H,m(n+1)+l}^{(3)}-M_{H,l}^{(3)}+t^{m}M_{H,mn+l}^{(3)}-t^{m}M_{H,l-m}^{(3)}\\
-M_{H,m(n+1)+l}^{(3)}\mu(m)+M_{H,l}^{(3)}\mu(m)\\
+M_{H,m(n+2)+l}^{(3)}-M_{H,l+m}^{(3)}
\end{array}
\right\rbrace
\end{equation}
and
\begin{equation}\label{p8}
\sum_{k=0}^{n}M_{h,mk+l}^{(3)}=\frac{1}{\sigma_{m}}\left\lbrace
\begin{array}{c}
M_{h,m(n+1)+l}^{(3)}-M_{h,l}^{(3)}+t^{m}M_{h,mn+l}^{(3)}-t^{m}M_{h,l-m}^{(3)}\\
-M_{h,m(n+1)+l}^{(3)}\mu(m)+M_{h,l}^{(3)}\mu(m)\\
+M_{h,m(n+2)+l}^{(3)}-M_{h,l+m}^{(3)}
\end{array}
\right\rbrace,
\end{equation}
where $\sigma_{m}=t^{m}(1+\alpha^{-m})+(1-\alpha^{m})(\omega_{1}^{m}+\omega_{2}^{m}-1)$ and $\mu(m)=\alpha^{m}+\omega_{1}^{m}+\omega_{2}^{m}$.
\end{thm}
\begin{proof}
(\ref{p7}): Let us take $$\left\lbrace
\begin{array}{c}
A_{H}=\frac{M_{H,2}^{(3)}-(\omega_{1}+\omega_{2})M_{H,1}^{(3)}+\omega_{1}\omega_{2}M_{H,0}^{(3)}}{(\alpha-\omega_{1})(\alpha-\omega_{2})},\\
B_{H}=\frac{M_{H,2}^{(3)}-(\alpha+\omega_{2})M_{H,1}^{(3)}+\alpha\omega_{2}M_{H,0}^{(3)}}{(\alpha-\omega_{1})(\omega_{1}-\omega_{2})},\ 
C_{H}=\frac{M_{H,2}^{(3)}-(\alpha+\omega_{1})M_{H,1}^{(3)}+\alpha\omega_{1}M_{H,0}^{(3)}}{(\alpha-\omega_{2})(\omega_{1}-\omega_{2})}
\end{array}
\right. .
$$
Then, we write
\begin{align*}
\sum_{k=0}^{n}M_{H,mk+l}^{(3)}&=\sum_{k=0}^{n}(A_{H}\alpha^{mk+l}-B_{H}\omega_{1}^{mk+l}+C_{H}\omega_{2}^{mk+l})\\
&=A_{H}\alpha^{l}\sum_{k=0}^{n}\alpha^{mk}-B_{H}\omega_{1}^{l}\sum_{k=0}^{n}\omega_{1}^{mk}+C_{H}\omega_{2}^{l}\sum_{k=0}^{n}\omega_{2}^{mk}\\
&=A_{H}\alpha^{l}\left(\frac{\alpha^{m(n+1)}-1}{\alpha^{m}-1}\right)-B_{H}\omega_{1}^{l}\left(\frac{\omega_{1}^{m(n+1)}-1}{\omega_{1}^{m}-1}\right)\\
&\ \ +C_{H}\omega_{2}^{l}\left(\frac{\omega_{2}^{m(n+1)}-1}{\omega_{2}^{m}-1}\right)\\
&=\frac{1}{\sigma_{m}}\left\lbrace
\begin{array}{c}
A_{H}\left(\alpha^{m(n+1)+l}-\alpha^{l}\right)\left(\omega_{1}^{m}\omega_{2}^{m}-(\omega_{1}^{m}+\omega_{2}^{m})+1\right)\\
- B_{H}\left(\omega_{1}^{m(n+1)+l}-\omega_{1}^{l}\right)\left(\alpha^{m}\omega_{2}^{m}-(\alpha^{m}+\omega_{2}^{m})+1\right)\\
+ C_{H}\left(\omega_{2}^{m(n+1)+l}-\omega_{2}^{l}\right)\left(\alpha^{m}\omega_{1}^{m}-(\alpha^{m}+\omega_{1}^{m})+1\right)
\end{array}
\right\rbrace,
\end{align*}
where $\sigma_{m}=t^{m}(1+\alpha^{-m})+(1-\alpha^{m})(\omega_{1}^{m}+\omega_{2}^{m}-1)$. After some algebra, we obtain
$$\sum_{k=0}^{n}M_{H,mk+l}^{(3)}=\frac{1}{\sigma_{m}}\left\lbrace
\begin{array}{c}
M_{H,m(n+1)+l}^{(3)}-M_{H,l}^{(3)}+t^{m}M_{H,mn+l}^{(3)}-t^{m}M_{H,l-m}^{(3)}\\
-M_{H,m(n+1)+l}^{(3)}\mu(m)+M_{H,l}^{(3)}\mu(m)\\
+M_{H,m(n+2)+l}^{(3)}-M_{H,l+m}^{(3)}
\end{array}
\right\rbrace,$$
where $\mu(m)=\alpha^{m}+\omega_{1}^{m}+\omega_{2}^{m}$.

(\ref{p8}): The proof is similar to the proof of (\ref{p7}).
\end{proof}
\section{The relationships between matrix sequences $M_{H,n}^{(3)}$ and $M_{h,n}^{(3)}$}

\begin{lem}\label{lem1}
For $m,n\in \mathbb{N}$, the third-order Horadam and generalized Tribonacci matrix sequences are conmutative. The following results hold.
\begin{equation}\label{p10}
M_{h,n}^{(3)}M_{h,m}^{(3)}=M_{h,m}^{(3)}M_{h,n}^{(3)}=M_{h,n+m}^{(3)},
\end{equation}
\begin{equation}\label{p11}
M_{H,n}^{(3)}M_{H,m}^{(3)}=M_{H,m}^{(3)}M_{H,n}^{(3)},
\end{equation}
\begin{equation}\label{p12}
M_{H,1}^{(3)}M_{h,n}^{(3)}=M_{H,n}^{(3)}M_{h,1}^{(3)}=M_{H,n+1}^{(3)},
\end{equation}
\begin{equation}\label{p13}
M_{H,n}^{(3)}M_{h,1}^{(3)}=M_{h,1}^{(3)}M_{H,n}^{(3)}=M_{H,n+1}^{(3)},
\end{equation}
\begin{equation}\label{p14}
M_{h,n}^{(3)}M_{H,n+1}^{(3)}=M_{H,2n+1}^{(3)}.
\end{equation}
\end{lem}
\begin{proof}
Here, we will just prove (\ref{p10}) and (\ref{p12}) since (\ref{p11}), (\ref{p13}) and (\ref{p14}) can be dealt with in the same manner. To prove Eq. (\ref{p10}), let us use the induction on $m$. If $m=0$, the proof is obvious since that $M_{h,0}^{(3)}$ is the identity matrix of order 3. Let us assume that Eq. (\ref{p10}) holds for all values $k$ less than or equal $m$. Now we have to show that the result is true for $m+1$:
\begin{align*}
M_{h,n+(m+1)}^{(3)}&=rM_{h,n+m}^{(3)}+sM_{h,n+m-1}^{(3)}+tM_{h,n+m-2}^{(3)}\\
&=rM_{h,n}^{(3)}M_{h,m}^{(3)}+sM_{h,n}^{(3)}M_{h,m-1}^{(3)}+tM_{h,n}^{(3)}M_{h,m-2}^{(3)}\\
&=M_{h,n}^{(3)}\left(rM_{h,m}^{(3)}+sM_{h,m-1}^{(3)}+tM_{h,m-2}^{(3)}\right)\\
&=M_{h,n}^{(3)}M_{h,m+1}^{(3)}.
\end{align*}
It is easy to see that $M_{h,n}^{(3)}M_{h,m}^{(3)}=M_{h,m}^{(3)}M_{h,n}^{(3)}$. Hence we obtain the result.

(\ref{p12}): To prove equation (\ref{p12}), we again use induction on $n$. Let $n=0$, we get $M_{H,1}^{(3)}M_{h,0}^{(3)}=M_{H,1}^{(3)}$. Let us assume that $M_{H,n}^{(3)}=M_{H,1}^{(3)}M_{h,n-1}^{(3)}$ is true for all values $k$ less than or equal $n$. Then,
\begin{align*}
M_{H,n+1}^{(3)}&=\left[
\begin{array}{ccc}
H_{n+2}^{(3)}& sH_{n+1}^{(3)}+tH_{n}^{(3)}& tH_{n+1}^{(3)} \\ 
H_{n+1}^{(3)}& sH_{n}^{(3)}+tH_{n-1}^{(3)}& tH_{n}^{(3)}\\ 
H_{n}^{(3)} & sH_{n-1}^{(3)}+tH_{n-2}^{(3)}&tH_{n-1}^{(3)}
\end{array}
\right]\\
&=\left[
\begin{array}{ccc}
H_{n+1}^{(3)}& sH_{n}^{(3)}+tH_{n-1}^{(3)}& tH_{n}^{(3)} \\ 
H_{n}^{(3)}&sH_{n-1}^{(3)}+tH_{n-2}^{(3)}& tH_{n-1}^{(3)}\\ 
H_{n-1}^{(3)} & sH_{n-2}^{(3)}+tH_{n-3}^{(3)}&tH_{n-2}^{(3)}
\end{array}
\right]\left[
\begin{array}{ccc}
r& s& t \\ 
1&0& 0\\ 
0& 1&0
\end{array}
\right]\\
&=M_{H,n}^{(3)}M_{h,1}^{(3)}\\
&=M_{H,1}^{(3)}M_{h,n-1}^{(3)}M_{h,1}^{(3)}\\
&=M_{H,1}^{(3)}M_{h,n}^{(3)}.
\end{align*}
Hence the result. 
\end{proof}

\begin{thm}\label{t5}
For $m,n\in \mathbb{N}$ the following properties hold.
\begin{equation}\label{p15}
M_{H,n}^{(3)}=bM_{h,n}^{(3)}+(c-rb)M_{h,n-1}^{(3)}+taM_{h,n-2}^{(3)},
\end{equation}
\begin{equation}\label{p16}
M_{H,n}^{(3)}=aM_{h,n+1}^{(3)}+(b-ra)M_{h,n}^{(3)}+(c-rb-sa)M_{h,n-1}^{(3)}.
\end{equation}
\end{thm}
\begin{proof}
Here, we will just prove (\ref{p15}) since (\ref{p16}) can be dealt with in the same manner. From Eq. (\ref{ec:5}) and $m=0$ in Eq. (\ref{e4}), we have
\begin{align*}
H_{n}^{(3)}&=h_{n}^{(3)}H_{1}^{(3)}+\left(sh_{n-1}^{(3)}+th_{n-2}^{(3)}\right)H_{0}^{(3)}+th_{n-1}^{(3)}H_{-1}^{(3)}\\
&=bh_{n}^{(3)}+\left(sh_{n-1}^{(3)}+th_{n-2}^{(3)}\right)a+(c-rb-sa)h_{n-1}^{(3)}\\
&=bh_{n}^{(3)}+(c-rb)h_{n-1}^{(3)}+tah_{n-2}^{(3)},
\end{align*}
So, if we consider the right-hand side of equation (\ref{p15}) and use Theorem \ref{t2}, we get
\begin{align*}
bM_{h,n}^{(3)}+&(c-rb)M_{h,n-1}^{(3)}+taM_{h,n-2}^{(3)}\\
&=b\left[
\begin{array}{ccc}
h_{n+1}^{(3)}& sh_{n}^{(3)}+th_{n-1}^{(3)}& th_{n}^{(3)} \\ 
h_{n}^{(3)}&sh_{n-1}^{(3)}+th_{n-2}^{(3)}& th_{n-1}^{(3)}\\ 
h_{n-1}^{(3)} & sh_{n-2}^{(3)}+th_{n-3}^{(3)}&th_{n-2}^{(3)}%
\end{array}%
\right]\\
&\ \ +(c-rb)\left[
\begin{array}{ccc}
h_{n}^{(3)}& sh_{n-1}^{(3)}+th_{n-2}^{(3)}& th_{n-1}^{(3)} \\ 
h_{n-1}^{(3)}&sh_{n-2}^{(3)}+th_{n-3}^{(3)}& th_{n-2}^{(3)}\\ 
h_{n-2}^{(3)} & sh_{n-3}^{(3)}+th_{n-4}^{(3)}&th_{n-3}^{(3)}%
\end{array}%
\right]\\
&\ \ +ta\left[
\begin{array}{ccc}
h_{n-1}^{(3)}& sh_{n-2}^{(3)}+th_{n-3}^{(3)}& th_{n-2}^{(3)} \\ 
h_{n-2}^{(3)}&sh_{n-3}^{(3)}+th_{n-4}^{(3)}& th_{n-3}^{(3)}\\ 
h_{n-3}^{(3)} & sh_{n-4}^{(3)}+th_{n-5}^{(3)}&th_{n-4}^{(3)}%
\end{array}%
\right]\\
&=\left[
\begin{array}{ccc}
H_{n+1}^{(3)}& sH_{n}^{(3)}+tH_{n-1}^{(3)}& tH_{n}^{(3)} \\ 
H_{n}^{(3)}&sH_{n-1}^{(3)}+tH_{n-2}^{(3)}& tH_{n-1}^{(3)}\\ 
H_{n-1}^{(3)} & sH_{n-2}^{(3)}+tH_{n-3}^{(3)}&tH_{n-2}^{(3)}%
\end{array}%
\right]\\
&=M_{H,n}^{(3)},
\end{align*}
as required in (\ref{p15}).
\end{proof}

\begin{thm}\label{t6}
For $m,n\in \mathbb{N}$, the following properties hold.
\begin{equation}\label{p20}
M_{h,m}^{(3)}M_{H,n+1}^{(3)}=M_{H,n+1}^{(3)}M_{h,m}^{(3)}=M_{H,m+n+1}^{(3)},
\end{equation}
\begin{equation}\label{p21}
\left(M_{H,n+1}^{(3)}\right)^{m}=\left(M_{H,1}^{(3)}\right)^{m}M_{h,mn}^{(3)}.
\end{equation}
\end{thm}
\begin{proof}
(\ref{p20}): Let us consider the left-hand side of equation (\ref{p20}) and Lemma \ref{lem1} and Theorem \ref{t5}. We have
\begin{align*}
M_{h,m}^{(3)}M_{H,n+1}^{(3)}&=M_{h,m}^{(3)}M_{H,1}^{(3)}M_{h,n}^{(3)}\\
&=M_{h,m}^{(3)}\left(bM_{h,1}^{(3)}+(c-rb)M_{h,0}^{(3)}+taM_{h,-1}^{(3)}\right)M_{h,n}^{(3)}\\
&=bM_{h,m+n+1}^{(3)}+(c-rb)M_{h,m+n}^{(3)}+taM_{h,m+n-1}^{(3)}\\
&=\left(bM_{h,1}^{(3)}+(c-rb)M_{h,0}^{(3)}+taM_{h,-1}^{(3)}\right)M_{h,m+n}^{(3)}.
\end{align*}
Moreover, from Eqs. (\ref{p12}) and (\ref{p13}) in Lemma \ref{lem1}, we obtain $$M_{h,m}^{(3)}M_{H,n+1}^{(3)}=M_{H,1}^{(3)}M_{h,m}^{(3)}M_{h,n}^{(3)}=M_{H,m+1}^{(3)}M_{h,m}^{(3)}.$$ Also, from Lemma \ref{lem1}, it is seen that $M_{h,m}^{(3)}M_{H,n+1}^{(3)}=M_{H,m+n+1}^{(3)}$ which finishes the proof of (\ref{p20}).

(\ref{p21}): To prove equation (\ref{p21}), let us follow induction steps on $m$. For $m=1$, the proof is clear by Lemma \ref{lem1}. Now, assume that it is true for all positive integers $m$, that is, $\left(M_{H,n+1}^{(3)}\right)^{m}=\left(M_{H,1}^{(3)}\right)^{m}M_{h,mn}^{(3)}$. 

Therefore, we have to show that it is true for $m+1$. If we multiply this $m$-th step by $M_{H,n+1}^{(3)}$ on both sides from the right, then we have
\begin{align*}
\left(M_{H,n+1}^{(3)}\right)^{m+1}&=\left(M_{H,1}^{(3)}\right)^{m}M_{h,mn}^{(3)}M_{H,n+1}^{(3)}\\
&=\left(M_{H,1}^{(3)}\right)^{m}M_{h,mn}^{(3)}M_{H,1}^{(3)}M_{h,n}^{(3)}\\
&=\left(M_{H,1}^{(3)}\right)^{m}M_{H,1}^{(3)}M_{h,mn}^{(3)}M_{h,n}^{(3)}\\
&=\left(M_{H,1}^{(3)}\right)^{m+1}M_{h,mn+n}^{(3)}\\
&=\left(M_{H,1}^{(3)}\right)^{m+1}M_{h,(m+1)n}^{(3)}
\end{align*}
which finishes the induction and gives the proof of (\ref{p21}).
\end{proof}

\begin{cor}
For $n\geq 0$, by taking $m=2$ and $m=3$ in the Eq. (\ref{p21}) given in Theorem \ref{t6}, we obtain
\begin{equation}\label{p22}
\left(M_{H,n+1}^{(3)}\right)^{2}=\left(M_{H,1}^{(3)}\right)^{2}M_{h,2n}^{(3)}=M_{H,1}^{(3)}M_{H,2n+1}^{(3)},
\end{equation}
\begin{equation}\label{p23}
\left(M_{H,n+1}^{(3)}\right)^{3}=\left(M_{H,1}^{(3)}\right)^{3}M_{h,3n}^{(3)}=\left(M_{H,1}^{(3)}\right)^{2}M_{H,3n+1}^{(3)}.
\end{equation}
\end{cor}

\begin{cor}
For $n \in \mathbb{N}_{0}$, we have the following result
\begin{equation}\label{p24}
\begin{aligned}
\left(H_{n+2}^{(3)}\right)^{2}+&s\left(H_{n+1}^{(3)}\right)^{2}+2tH_{n}^{(3)}H_{n+1}^{(3)}\\
&=\left\lbrace \begin{array}{c}(c^{2}+sb^{2}+2tab)h_{2n+1}^{(3)}\\+(b^{2}(t-rs)+2tac+2sbc-2rtab)h_{2n}^{(3)}\\+t(ta^{2}-rb^{2}+2bc)h_{2n-1}^{(3)}\end{array}\right\rbrace\\
&=cH_{2n+2}^{(3)}+(sb+ta)H_{2n+1}^{(3)}+tbH_{2n}^{(3)}.
\end{aligned}
\end{equation}
\end{cor}
\begin{proof}
The proof can be easily seen by the coefficient in the first row and column of the matrix $$\left(M_{H,n+1}^{(3)}\right)^{2}=\left(M_{H,1}^{(3)}\right)^{2}M_{h,2n}^{(3)}=M_{H,1}^{(3)}M_{H,2n+1}^{(3)}$$ in Eq. (\ref{p22}) and $M_{H,1}^{(3)}$ from Eq. (\ref{d1}).
\end{proof}
\section{Conclusions}
In this paper, we study a generalization of the Horadam and generalized Fibonacci matrix sequences. Particularly, we define the third-order Horadam and generalized Tribonacci matrix sequences, and we find some combinatorial identities. It would be interesting to introduce the higher order Horadam and generalized Fibonacci matrix sequences. Further investigations for these and other methods useful in discovering identities for the higher order Horadam and generalized Fibonacci sequences will be addressed in a future paper.



\begin{thebibliography}{00}

\bibitem{Ba}
P. Barry, \emph{Triangle geometry and Jacobsthal numbers}, Irish Math. Soc. Bull. 51 (2003), 45--57.
\bibitem{Cer} 
G. Cerda-Morales, \emph{Identities for Third Order Jacobsthal Quaternions}, Advances in Applied Clifford Algebras 27(2) (2017), 1043--1053.
\bibitem{Cer1} 
G. Cerda-Morales, \emph{On a Generalization of Tribonacci Quaternions}, Mediterranean Journal of Mathematics 14:239 (2017), 1--12.
\bibitem{Ci1} 
H. Civciv and R. T\"urkmen, \emph{On the $(s, t)$-Fibonacci and Fibonacci matrix sequences}, Ars Combinatoria 87 (2008), 161--173.
\bibitem{Ci2} 
H. Civciv and R. T\"urkmen, \emph{Notes on the $(s, t)$-Lucas and Lucas matrix sequences}, Ars Combinatoria 89 (2008), 271--285.
\bibitem{Cook-Bac} 
C. K. Cook and M. R. Bacon, \textit{Some identities for Jacobsthal and Jacobsthal-Lucas numbers satisfying higher order recurrence relations}, Annales Mathematicae et Informaticae 41 (2013), 27--39.
\bibitem{Sha} 
A. G. Shannon  and A. F. Horadam, \textit{Some Properties of Third-Order Recurrence Relations}, The Fibonacci Quarterly 10(2) (1972), 135--146.
\bibitem{Lar1} 
P. J. Larcombe, \textit{Horadam sequences: a survey update and extension}, Bull. I.C.A. 80 (2017), 99--118.
\bibitem{Lar2} 
P. J. Larcombe, O. D. Bagdasar and E. J. Fennessey, \textit{Horadam sequences: a survey}, Bull. I.C.A. 67 (2013), 49--72.
\bibitem{Wa} 
M. E. Waddill and L. Sacks, \textit{Another Generalized Fibonacci Sequence}, The Fibonacci Quarterly 5(3) (1967), 209--222.
\end{thebibliography}
\end{document}